\documentclass[a4paper]{article}
\usepackage{amsmath}
\usepackage{amssymb}
\usepackage{amsfonts}

\newtheorem{theorem}{Theorem}[section]

\newtheorem{corollary}[theorem]{Corollary}

\newtheorem{example}[theorem]{Example}

\newtheorem{lemma}[theorem]{Lemma}

\newtheorem{proposition}[theorem]{Proposition}
\newtheorem{remark}[theorem]{Remark}

\newenvironment{proof}[1][Proof]{\noindent\textbf{#1.} }{\ \rule{0.5em}{0.5em}}
\input{tcilatex}
\begin{document}

\title{On Artinian Rings with Restricted Class of Injectivity Domains}
\author{\textsf{P\i nar Aydo\u{g}du}\thanks{%
Corresponding author : paydogdu@hacettepe.edu.tr (P. Aydo\u{g}du)}\qquad 
\textsf{and\qquad B\"{u}lent Sara\c{c}} \\
{\small Department of Mathematics, Hacettepe University, Beytepe 06800
Ankara / Turkey}\textsf{\ }}
\date{}
\maketitle

\begin{abstract}
In a recent paper of Alahmadi, Alkan and L\'{o}pez--Permouth, a ring $R$ is
defined to have no (simple) middle class if the injectivity domain of any
(simple) $R$--module is the smallest or largest possible. Er, L\'{o}%
pez--Permouth and S\"{o}kmez use this idea of restricting the class of
injectivity domains to classify rings, and give a partial characterization
of rings with no middle class. In this work, we continue the study of the
property of having no (simple) middle class. We give a structural
description of right Artinian right nonsingular rings with no right middle
class. We also give a characterization of right Artinian rings that are not $%
SI$ to have no middle class, which gives rise to a full characterization of
rings with no middle class. Furthermore, we show that commutative rings with
no middle class are those Artinian rings which decompose into a sum of a
semisimple ring and a ring of composition length two. Also, Artinian rings
with no simple middle class are characterized. We demonstrate our results
with several examples.
\end{abstract}

\section{\protect\bigskip Introduction}

\hspace{0.5cm}Throughout this paper, our rings will be associative rings
with identity, and modules will be unital right modules, unless otherwise
stated. For any ring $R,\ \func{Mod}$--$R$ will denote the category of all
right $R$--modules$.\medskip $

Let $R$ be a ring. Recall that an $R$--module $M$ is\emph{\ injective
relative to an }$R$\emph{--module }$N$ (or, $M$ \emph{is }$N$\emph{%
--injective}) if, for any submodule $K$ of $N,$ any $R$ --homomorphism $f:$ $%
K\rightarrow M$ extends to some member of $Hom_{R}(N,M).$ It is evident that
every module is injective relative to semisimple modules. Thus, for any $R$%
--module $M,$ \emph{the injectivity domain }$\mathfrak{In}^{-1}(M)=\{N\in 
\func{Mod}$--$R:M$ is $N$--injective$\}$ of $M$ contains all semisimple
right $R$--modules. In \cite{AAL}, Alahmadi, Alkan and L\'{o}pez--Permouth
initiated the study of \emph{poor modules,} namely modules whose injectivity
domains consist only of semisimple modules in $\func{Mod}$--$R.$ They
consider rings over which every right module is either injective or poor,
and refer such rings as \emph{having no right middle class.} The study of
rings with no middle class has a growing interest in recent years (see \cite%
{AAL}, \cite{Er}, \cite{Nil}, and \cite{lopez-simental}).\emph{\medskip }

In \cite{Er}, Er, L\'{o}pez--Permouth and S\"{o}kmez studied the rings with
no right middle class and gave a partial characterization of such rings. The
following two theorems summarize the results on rings with no right middle
class obtained in \cite{Er}. To simplify the statements, we assume that the
ring $R$ is not semisimple Artinian. All statements can be made to fit that
possibility by setting $T=0.\medskip $

\noindent \textbf{Theorem 1 }Let $R$ be a right $SI$--ring. Then $R$ has no
right middle class if and only if $R\cong S\oplus T,$ where $S$ is
semisimple Artinian and $T$ is either

$(i)$ Morita equivalent to a right $PCI$--domain or

$(ii)$ an indecomposable ring with homogenenous essential right socle
satifying one of the following equivalent conditions (where $Q$ is the
maximal right quotient ring of $T$):

\qquad $(a)$ Non-semisimple quasi--injective right $T$--modules are
injective.

$\qquad (b)$ Proper essential submodules of $Q_{T}$ are poor.

\qquad $(c)$ For any submodule $A$ of $Q_{T}$ containing $Soc(T_{T})$
properly, $QA=Q.$

Those rings of type $(ii)$ are either right Artinian or right $V$--rings and
have a unique simple singular right $T$--module up to isomorphism.\medskip

\noindent \textbf{Theorem 2} Let $R$ be a ring with no right middle class
which is not right $SI$. Then $R\cong S\oplus T,$ where $S$ is semisimple
Artinian and $T$ is an indecomposable right Artinian ring satisfying the
following conditions:

$(i)$ $soc(T_{T})=Z(T_{T})=J(T),$

$(ii)$ $T$ has homogeneous right socle, and

$(iii)$ there is a unique non--injective simple right $T$--module up to
isomorphism.

In this case $T$ is either a $QF$--ring with $J(T)^{2}=0,$ or poor as a
right module. Conversely, if $T$ is a $QF$--ring with homogeneous right
socle and $J(T)^{2}=0,$ then $T$ has no right middle class.\medskip

Note that the authors of \cite{Er} could not reverse this implication to
show that the conditions $(i)$--$(iii)$ in Theorem 2 above are sufficient as
well as necessary. As a matter of fact, we show in our work that there exist
rings satisfying conditions $(i)$--$(iii)$ in Theorem 2 above which are poor
as a right module over itself and do have right middle class (see Examples %
\ref{sey1} and \ref{sey2}). We also give a complete characterization of
non--SI--rings with no right middle class (see Theorem \ref{non-SI}).

A characterization of right Artinian rings with no right middle class was
given in Corollary 3.2 of \cite{lopez-simental}. Using that result in
conjunction with those of \cite{Er}, we have the following two complete
characterizations:\medskip

\noindent \textbf{Theorem 3} Let $R$ be any ring. Then $R$ has no right
middle class if and only if $R\cong S\oplus T,$ where $S$ is semisimple
Artinian and $T$ satisfies one of the following conditions:

$(i)$ $T$ is Morita equivalent to a right $PCI$--domain, or

$(ii)$ $T$ is a right $SI$ right $V$--ring with the following properties:

\qquad $(a)$ $T$ has essential homogeneous right socle and

\qquad $(b)$ for any submodule $A$ of $Q_{T}$ which does not contain the
right socle of $T$ properly, $QA=Q,$ where $Q$ is the maximal right quotient
ring of $T,$ or

$(iii)$ $T$ is a right Artinian ring whose Jacobson radical properly
contains no nonzero ideals.\medskip

\noindent \textbf{Theorem 4} Let $R$ be any ring. Then $R$ has no right
middle class if and only if $R\cong S\oplus T,$ where $S$ is semisimple
Artinian and $T$ satisfies one of the following conditions:

$(i)$ $T$ is Morita equivalent to a right $PCI$--domain, or

$(ii)$ $(a)$ $T$ is a right Artinian ring or a right $SI$ right $V$--ring
with homogeneous essential right socle, and

\qquad $(b)$ every nonsemisimple quasi--injective right $T$--module is
injective.\medskip

In the process of studying these rings, various necessary or sufficient
conditions are presented in \cite{Er}. For instance,

\textbf{(P1) }$R$ has homogeneous right socle, and\smallskip

\textbf{(P2) }there is a unique simple singular right $R$--module up to
isomorphism\newline
are necessary conditions for a nonsemisimple indecomposable right $SI$--ring 
$R$ to have no right middle class (see \cite[Theorem 2]{Er}). Likewise, in 
\cite[Proposition 6]{Er}, it is shown that right Artinian right $SI$--rings
with homogeneous right socle anda unique local module of length two up tp
isomorphism must have no right middle class. We show here that \textbf{(P1)}
and \textbf{(P2)} are not sufficient while the condition that the ring has a
unique local module of length two up to isomorphism is not necessary (see
Examples \ref{ex2}(i) and \ref{example_local module of length 2}). \medskip

It is shown, in \cite[Corollary 5]{Er}, that if $R$ is an indecomposable
right nonsingular right Artinian ring with no right middle class, then $R$
is isomorphic to a formal triangular matrix ring of the form $\left( 
\begin{array}{cc}
S & 0 \\ 
A & S^{\prime }%
\end{array}%
\right) ,$ where $S$ and $S^{\prime }$ are simple Artinian rings and $A$ is
an $S^{\prime }$--$S$--bimodule. Using the theory of Morita equivalences, we
see that such rings simplifies to formal triangular matrix rings of the form 
$\left( 
\begin{array}{cc}
D & 0 \\ 
\mathbb{M}_{n\times 1}(D) & D^{\prime }%
\end{array}%
\right) ,$ where $D$ is a division ring and $D^{\prime }$ is a division
subring of $\mathbb{M}_{n}(D)$ for some positive integer $n$ (Theorem \ref%
{thmf}). We also prove that certain conditions on $D^{\prime }$
characterizes these triangular rings to have no right middle class which
yields a general characterization for right nonsingular right Artinian rings
to have no right middle class (Theorem \ref{thm1}). This result also enables
us to produce many interesting examples of right nonsingular right Artinian
rings with no right middle class.\medskip

It is also known from \cite[Corollary 6]{Er} that if $R$ is an
indecomposable right Artinian ring with no right middle class which is not
right $SI$, then $R$ is isomorphic to a formal triangular matrix ring of the
form $\left( 
\begin{array}{cc}
\mathbb{M}_{n}(A) & 0 \\ 
X & B%
\end{array}%
\right) ,$ where $A$ is a (nonsemisimple) local right Artinian ring, $B$ is
a semisimple Artinian ring, and $X$ is a $B$--$\mathbb{M}_{n}(A)$--bimodule.
As a matter of fact, a right Artinian ring which is not right $SI$ has no
right middle class if and only if $R\cong S\oplus \mathbb{M}_{n}(A)$ where $%
S $ is semisimple Artinian and $A$ is a local right Artinian ring whose
Jacobson radical properly contains no nonzero ideals (Theorem \ref{non-SI}).
\medskip

In Section 3, we restrict our attention to only simple modules, and consider
rings whose simple right modules are either injective or poor. Such rings
are said to have no simple middle class (see \cite{AAL}). We give necessary
and sufficient conditions for a right Artinian ring to have no simple middle
class.\medskip

The last section of our paper is concerned with the property of having no
(simple) middle class in the commutative setting. We give a complete
description of commutative rings with no middle class. In particular, we see
that a commutative ring with no middle class is Artinian. We conclude our
work with a characterization of commutative Noetherian rings to have no
simple middle class.\medskip

Recall that a ring is said to be a right $V$--ring if every simple right
module is injective. As a generalization of right $V$--rings, right $GV$%
--rings were introduced by Ramamurthi and Rangaswamy in \cite{GV-rings}. A
ring is called right $GV$ if every simple singular module is injective, or
equivalently, every simple module is either injective or projective. We call
a ring right $SI$ if every singular right module is injective (see\cite%
{Extending Modules}). Note that semilocal right $GV$--rings are right $SI.$%
\medskip

If $M$ is an $R$ --module, then $E(M),$ $J(M),$ $Z(M)$ and $Soc(M)$ will
respectively denote the injective hull, Jacobson radical, the singular
submodule and the socle of $M.$ We will use the notations $\leq $ and $\leq
_{e}$ in order to indicate submodules and essential submodules,
respectively. For a module with a composition series, $cl(M)$ stands for the
composition length of $M.$ The ring of $n\times n$ matrices over a ring $R$
will be denoted by $\mathbb{M}_{n}(R).$ The notation $A[i,j]$ will be used
to indicate the $(i,j)$--th entry of a matrix $A.$ We use $e_{ij}$ to
designate the standard matrix unit of $\mathbb{M}_{n}(R)$ with $1$ in the $%
(i,j)$--th entry and zeros elsewhere. For any unexplained terminology, we
refer the reader to \cite{Facchini} and \cite{Lam}.

\section{Artinian Rings with No Middle Class}

\hspace{0.5cm}In \cite[Corollary 5]{Er}, Er, L\'{o}pez--Permouth, and S\"{o}%
kmez proved that if $R$ is an indecomposable right nonsingular right
Artinian ring with no right middle class, then $R$ is isomorphic to a formal
triangular matrix ring of the form $\left( 
\begin{array}{cc}
S & 0 \\ 
A & S^{\prime }%
\end{array}%
\right) ,$ where $S$ and $S^{\prime }$ are simple Artinian rings and $A$ is
an $S^{\prime }$--$S$--bimodule. With the following theorem, we see that to
determine when such rings have no right middle class, it is enough to
consider formal triangular matrix rings of the much simpler form $\left( 
\begin{array}{cc}
D & 0 \\ 
\mathbb{M}_{n\times 1}(D) & D^{\prime }%
\end{array}%
\right) ,$ where $D$ is a division ring and $D^{\prime }$ is a division
subring of $\mathbb{M}_{n}(D)$ for some positive integer $n$

\begin{theorem}
\label{thmf}If $R$ is a right Artinian right $SI$ ring satisfying the
properties $(P1)$ and $(P2)$, then it is Morita equivalent to a formal
triangular matrix ring of the form 
\begin{equation*}
\left( 
\begin{array}{cc}
D & 0 \\ 
\mathbb{M}_{n\times 1}(D) & D^{\prime }%
\end{array}%
\right) ,
\end{equation*}%
where $D$ is a division ring and $D^{\prime }$ is a division subring of $%
\mathbb{M}_{n}(D)$.
\end{theorem}

\begin{proof}
Since $R$ is a right Artinian ring, there exists a complete set $%
\{e_{1},\ldots ,e_{l},f_{1},\ldots ,f_{m}\}$ of local orthogonal idempotents
such that the $e_{i}R$'s are simple and $f_{i}R$'s are nonsemisimple local.
Then $R=e_{1}R\oplus \cdots \oplus e_{l}R\oplus f_{1}R\oplus \cdots \oplus
f_{m}R$. By $(P2)$, $f_{i}R\cong f_{j}R$ for all $i,j$.\ Moreover, since $%
e_{i}R$ is nonsingular, there is no nonzero $R$--homomorphism from $f_{j}R$
into $e_{i}R,$ and so $e_{i}Rf_{j}=0$ for all $i,j$. As $R$ is right $SI$, $%
R/Soc(R_{R})$ is semisimple Thus, $J(R)\leq Soc(R_{R})$. This gives that $%
J(f_{i}R)=Soc(f_{i}R)$. We set $e=e_{1}+f_{1}$. Then $eR=e_{1}R\oplus f_{1}R$%
. We shall prove that $e$ is a full idempotent of $R$, i.e., $ReR=R$. It is
clear that $e_{1}$, $f_{1}\in ReR$. Since $R$ has homogeneous right socle, $%
e_{i}Re_{1}R\neq 0$ for all $i=1,\ldots ,l$. Then $e_{i}\in
e_{i}R=e_{i}Re_{1}R\leq ReR$ for all $i=1,\ldots ,l$. Now, assume that $%
f_{k}\notin f_{k}Rf_{1}R$. Then $f_{k}Rf_{1}R\leq Soc(f_{k}R)$. Since $R$
has homogeneous right socle and $e_{i}Rf_{j}=0$ for all $i,j$, $%
f_{k}Rf_{1}Rf_{1}=0$. Let $u:f_{1}R\rightarrow f_{k}R$ be an isomorphism of
right $R$--modules. Then $u(f_{1}Rf_{1}Rf_{1})=u(f_{1})Rf_{1}Rf_{1}\leq
f_{k}Rf_{1}Rf_{1}=0$, and so $f_{1}\in f_{1}Rf_{1}Rf_{1}=0$, a
contradiction. It follows that $f_{k}\in f_{k}Rf_{1}R\leq ReR$ for all $%
k=1,\ldots ,m$. Therefore, $ReR=R$.

Now let $\alpha $ be an endomorphism on $E(f_{1}R)$, and let $\alpha
^{\prime }$ be the restriction of $\alpha $ to $Soc(f_{1}R)$. Suppose $%
Soc(f_{1}R)=B_{1}\oplus \cdots \oplus B_{n}$, where $B_{1},\ldots ,B_{n}$
are simple right $R$--modules isomorphic to $e_{1}R$, and let $%
v_{k}:B_{k}\rightarrow e_{1}R$ be an isomorphism. Set $\alpha
_{k}=i_{k}v_{k}^{-1}$ and $\beta _{k}=v_{k}\pi _{k}$, where $i_{k}$ is the
natural embedding of $B_{k}$ into $Soc(f_{1}R),$ and $\pi _{k}$ the natural
projection of $Soc(f_{1}R)$ onto $B_{k}$. Then the correspondence%
\begin{equation*}
\alpha \overset{\varphi }{\longleftrightarrow }(\beta _{i}\alpha ^{\prime
}\alpha _{j})_{n\times n}
\end{equation*}%
between $End(E(f_{1}R))$ and $\mathbb{M}_{n}(End(e_{1}R))$ gives a ring
isomorphism. Moreover, since $E(f_{1}R)$ is nonsingular and $%
E(f_{1}R)/Soc(f_{1}R)$ is singular, this correspondence yields an embedding
of $End(f_{1}R)$ into $\mathbb{M}_{n}(End(e_{1}R))$. We denote $\varphi
(End(f_{1}R))$ by $D^{\prime }$. Note that $D^{\prime }$ is a division ring
in view of the proof of \cite[Corollary 5]{Er}. Now it is routine to check
that the mapping%
\begin{equation*}
\begin{tabular}{rll}
$eRe\cong \left( 
\begin{array}{cc}
End(e_{1}R) & 0 \\ 
Hom(e_{1}R,f_{1}R) & End(f_{1}R)%
\end{array}%
\right) $ & $\longrightarrow $ & $\left( 
\begin{array}{cc}
D & 0 \\ 
\mathbb{M}_{n\times 1}(D) & D^{\prime }%
\end{array}%
\right) ,\bigskip $ \\ 
$\left( 
\begin{array}{rr}
u_{1} & 0 \\ 
u_{2} & u_{3}%
\end{array}%
\right) $ & $\longmapsto $ & $\left( 
\begin{array}{cc}
u_{1} & 0 \\ 
(\beta _{k}u_{2})_{n\times 1} & \varphi (u_{3})%
\end{array}%
\right) ,$%
\end{tabular}%
\end{equation*}%
where $D$ denotes $End(e_{1}R),$ is an isomorphism of rings. This completes
the proof.\bigskip
\end{proof}

From now on, we will denote the formal triangular matrix ring 
\begin{equation*}
\left( 
\begin{array}{cc}
D & 0 \\ 
\mathbb{M}_{n\times 1}(D) & D^{\prime }%
\end{array}%
\right)
\end{equation*}%
by $(D,D^{n},D^{\prime }).$

Let $R=(D,D^{n},D^{\prime })$, where $D$ is a division ring and $D^{\prime }$
is a division subring of $\mathbb{M}_{n}(D)$. Then $R=(D,0,0)\oplus
(0,D^{n},D^{\prime })$, where $(D,0,0)$ is a simple right ideal and $%
(0,D^{n},D^{\prime })$ is a local right ideal with the maximal submodule $%
(0,D^{n},0)$.

Define $(D,\mathbb{M}_{n}(D))_{i}$ as the set of ordered pairs $(a,A)$,
where $a\in D$, $A\in \mathbb{M}_{n}(D)$, and a scalar multiplication

\begin{equation*}
(a,A)(a_{1},(x_{k}),A_{1})=(aa_{1}+A^{(i)}(x_{k}),AA_{1}),
\end{equation*}%
where $(a_{1},(x_{k}),A_{1})\in R$ and $A^{(i)}$ denotes the $i$th row of $A$%
. It is easy to see that $(D,D^{\prime })_{i}$ is a right $R$--submodule of $%
(D,\mathbb{M}_{n}(D))_{i}$. Also, if $B_{i}=\oplus _{j\neq i}e_{j}D$, where $%
\{e_{1},\ldots ,e_{n}\}$ is the natural basis for $D^{n}$ over $D$, then $%
(D,D^{\prime })_{i}$ is isomorphic to $\dfrac{(0,D^{n},D^{\prime })}{%
(0,B_{i},0)}$.

\begin{lemma}
\label{lem1} $u:(D,D^{\prime })_{i}\rightarrow (D,\mathbb{M}_{n}(D))_{j}$ is
a nonzero $R$--homomorphism if and only if there exist $d_{0}\in D$ and $%
A_{0}\in \mathbb{M}_{n}(D)$ such that $A_{0}[j,k]=\delta _{ki}d_{0}$ ($%
k=1,\ldots ,n$) and $u(d,A)=(d_{0}d,A_{0}A)$ for all $(d,A)\in (D,D^{\prime
})_{i}$.
\end{lemma}

\begin{proof}
Straightforward.
\end{proof}

\begin{lemma}
\label{no middle class & factor rings}\cite[Lemma 1]{Er} The property of
having no (simple) middle class is inherited by factor rings.
\end{lemma}

\begin{lemma}
\label{middle class & ring direct sum}Suppose a ring $R=S\oplus T$ is a
direct sum of two rings $S$ and $T,$ where $S$ is semisimple. Then $R$ has
no (simple) middle class if and only if $T$ has no (simple) middle class.
\end{lemma}

\begin{proof}
Let $M$ be an $N$--injective right $R$--module, where $N$ is cyclic and
nonsemisimple. Then $N$ is isomorphic to a direct sum $S/A\oplus T/B$ of
right $R$--modules for some right ideals $A$ and $B$ contained in $S$ and $T$%
, respectively. Note that $T/B$ is not semisimple (as both $R$-- and $T$%
--modules). Since $M=MS\oplus MT,$ $MT$ is $(T/B)$--injective as both $R$--
and $T$--modules. By assumption, $MT$ is an injective right $T$--module.
However, it is not difficult to see that it is also injective as an $R$%
--module. We may also show, in a similar way, that $MS$ is an injective
right $R$--module. This gives that $M$ is an injective $R$--module. Thus we
established the sufficiency part. The necessity is obvious by the above
\linebreak lemma.\bigskip
\end{proof}

Let $S$ be a subring of a ring $R$ and $u$ a unit in $R$. Obviously, $%
uSu^{-1}$ is a subring of $R$ isomorphic to $S$ as a ring. We call $uSu^{-1}$
\emph{a conjugate ring of }$S$\emph{\ in }$R$.

\begin{theorem}
\label{thm1} Let $R$ be a right nonsingular right Artinian ring. Then $R$
has no right middle class if and only if $R\cong S\oplus T,$ where $S$ is a
semisimple Artinian ring and $T$ is zero or Morita equivalent to a formal
triangular matrix ring of the form $(D,D^{n},D^{\prime })$ where $D$ is a
division ring and $D^{\prime }$ is a division subring of $\mathbb{M}_{n}(D)$
such that for each conjugate ring $U$ of $D^{\prime }$ in $\mathbb{M}_{n}(D)$%
, the set of $i$--th rows of elements in $U$ span $D^{n}$ as a left $D$%
--space for every $i=1,\ldots ,n$.
\end{theorem}

\begin{proof}
Assume first that $R$ has no right middle class. By \cite[Theorem 2]{Er}, $%
R\cong S\oplus T,$ where $S$ is a semisimple Artinian ring and $T$ is zero
or a right Artinian right SI ring satisfying the properties $(P1)$ and $(P2)$%
. Suppose $T$ is not zero. Then by Theorem \ref{thmf}, $T$ is Morita
equivalent to a formal triangular matrix ring of the form $%
(D,D^{n},D^{\prime }),$ where $D$ is a division ring and $D^{\prime }$ is a
division subring of $\mathbb{M}_{n}(D).$ Since the property of having no
right middle class is a Morita invariant property (as remarked, for example,
in \cite{Er} before Proposition 5), the ring $(D,D^{n},D^{\prime })$ has no
right middle class. Note that $(D,0)_{1}\leq _{e}(D,\oplus
_{j=1}^{n}De_{ij})_{1}$ for all $i=1,\ldots ,n$. By Lemma~\ref{lem1}, $%
(D,De_{1k}D^{\prime })_{1}$ is $(D,D^{\prime })_{k}$--injective as a right $%
R $--module for all $k=1,\ldots ,n$. Since $R$ has no right middle class and 
$(D,D^{\prime })_{k}$ is nonsemisimple, $(D,De_{1k}D^{\prime })_{1}$ must be
injective for all $k=1,\ldots ,n$. On the other hand, we have $(D,0)_{1}\leq
(D,De_{1k}D^{\prime })_{1}\leq (D,\oplus _{j=1}^{n}De_{1j})_{1},$ which
gives that $De_{1k}D^{\prime }=\oplus _{j=1}^{n}De_{1j}$ for all $k=1,\ldots
,n$. It therefore follows that the set of the $k$--th rows of elements of $%
D^{\prime }$ span $D^{n}$ as a left $D$--space for all $k=1,\ldots ,n$. Now
let $U=uD^{\prime }u^{-1}$ for some unit $u$ in $\mathbb{M}_{n}(D)$.
Obviously, $R\cong (1,0,u)R(1,0,u^{-1})=(D,D^{n},U)$. Then $(D,D^{n},U)$ has
no right middle class. Repeating the above arguments, we complete the proof
of the necessity part.

For the sufficiency, it is enough, by Lemma \ref{middle class & ring direct
sum}, to show that if, for each conjugate ring $U$ of $D^{\prime }$ in $%
\mathbb{M}_{n}(D)$, the set of $i$--th rows of elements in $U$ span $D^{n}$
as a left $D$--space for every $i=1,\ldots ,n,$ then the ring $%
(D,D^{n},D^{\prime })$ has no right middle class. Assume the contrary, i.e.,
assume that $(D,D^{n},D^{\prime })$ has a right middle class. Note that the
maximal right quotient ring of $R$ is $Q=\mathbb{M}_{n+1}(D)$. By \cite[%
Proposition 8]{Er}, there exists $X\leq Q_{R}$ which contains the right
socle of $(D,D^{n},D^{\prime })$ properly such that $QX\neq Q$. Since the
right socle of the ring $(D,D^{n},D^{\prime })$ is $(D,D^{n},0)$, there
exists a nonzero right $D^{\prime }$--submodule $Y$ of $\mathbb{M}%
_{(n+1)\times n}(D)$ such that 
\begin{equation*}
X=\left( 
\begin{array}{cc}
\begin{array}{c}
D \\ 
D^{n}%
\end{array}
& Y%
\end{array}%
\right) .
\end{equation*}%
One can observe that if, for each $i=1,\ldots ,n,$ there exists an element
of $Y$ (depending on $i$) whose $i$--th column has a nonzero entry, then $%
QX=Q$. Thus, there exists $j$ such that the $j$--th column of each element
of $Y$ is zero. Without loss of generality, we may assume $j=1$. Then there
exist elements $d_{1},\ldots ,d_{n-1}$ of $D$ which are not all zero such
that $d_{1}A[2,1]+\cdots +d_{n-1}A[n,1]=0$ for all $A\in D^{\prime }$. We
may choose $d_{n-1}\neq 0$. Let $B=e_{11}+(%
\sum_{i=2}^{n}e_{ii}+d_{i-1}e_{ni}).$ Obviously, $B$ is invertible in $%
\mathbb{M}_{n}(D)$, and all elements of $BD^{\prime }B^{-1}$ have zero in
the $(n,1)$--th entry. It follows that the $n$--th rows of elements of $%
BD^{\prime }B^{-1}$ cannot span $D^{n}$. This completes the proof.
\end{proof}

\begin{corollary}
\label{cor1}(i) Let $D$ be a division ring and $D^{\prime }$ a division
subring of $D.$ Then the ring $\left( 
\begin{array}{cc}
D & 0 \\ 
D & D^{\prime }%
\end{array}%
\right) $ has no right middle class.

(ii) Let $D$ be a division ring and $D^{\prime }$ a division subring of $%
\mathbb{M}_{2}(D)$. Then the ring $R=(D,D^{2},D^{\prime })$ has no right
middle class if and only if the set of the $i$th rows of elements of $%
D^{\prime }$ span $D^{2}$ for $i=1,2$.
\end{corollary}

\begin{proof}
(i) Clear by Theorem \ref{thm1}. (ii) If we take $n=2$ in the proof of
Theorem \ref{thm1}, then we deduce that when $R$ has right middle class,
there exists a nonzero element $d\in D$ such that $dA[2,1]=0$ for all $A\in
D^{\prime }.$ This contradicts the fact that the set of the second rows of
elements of $D^{\prime }$ span $D^{2}.\bigskip $
\end{proof}

It is shown in \cite[Proposition 6]{Er} that if $R$ is a right Artinian
right $SI$--ring with homogeneous right socle and a unique local module of
length two up to isomorphism, then $R$ has no right middle class. Example %
\ref{ex2} below shows that the converse of this fact is not true in general.
Before the example, we need the following proposition.

\begin{proposition}
\label{prop1} Let $D$ be a division ring and $D^{\prime }$ be a division
subring of $\mathbb{M}_{n}(D).$ Let $R=(D,D^{n},D^{\prime })$ and let $%
\mathfrak{R}$ be the set of the first rows of all elements of $D^{\prime }$.
Then $R$ has a unique local right $R$--module of length two up to
isomorphism if and only if $\bigcup_{(x_{i})\in \mathfrak{R}}D(x_{i})=D^{n}$.
\end{proposition}

\begin{proof}
Let $M$ be a local right $R$--module of length two. Then there exists an
epimorphism $g:R\rightarrow M$, and since $M$ is local, we must have $%
g(0,D^{n},D^{\prime })=M$. It follows that there exists a maximal submodule $%
A$ of $(0,D^{n},0)$ such that $A=Ker(g)$. Thus the unique simple submodule
of $M$, say $N$, is isomorphic to $(D,0,0)$.

Observe that $(D,\bigoplus_{j=1}^{n}De_{1j})_{1}$ is injective relative to $%
(D,D^{\prime })_{i}$ for each $i=1,\ldots ,n,$ by Lemma \ref{lem1}. Since $R$
can be embedded into the sum $(D,0,0)\oplus (D,D^{\prime })_{1}\oplus \cdots
\oplus (D,D^{\prime })_{n},$ we get that $(D,\bigoplus_{j=1}^{n}De_{1j})_{1}$
is an injective $R$--module. It follows that $(D,%
\bigoplus_{j=1}^{n}De_{1j})_{1}$ is the injective hull of the local right $R$%
--module $(D,e_{11}D^{\prime })_{1}$. There is an isomorphism from $N$ to $%
(D,0)_{1}$ which extends to a homomorphism $f$ from $M$ into $%
(D,\bigoplus_{j=1}^{n}De_{1j})_{1}$. Since $(D,%
\bigoplus_{j=1}^{n}De_{1j})_{1}$ is nonsingular and $M/N$ is singular, $f$
must be a monomorphism. It therefore follows that every local right $R$%
--module of length two can be embedded into $(D,%
\bigoplus_{j=1}^{n}De_{1j})_{1}$.

Let $M^{\prime }$ be a local submodule of $(D,%
\bigoplus_{j=1}^{n}De_{1j})_{1} $ of length two. Then $M^{\prime }=(D,X)_{1}$
for a right $D^{\prime }$--subspace $X$ of $\bigoplus_{j=1}^{n}De_{1j}.$
Since $M^{\prime }$ is local, $X$ must be one--dimensional as a $D^{\prime }$%
--space. This shows that any local right $R$--submodule of length two in $%
(D,\bigoplus_{j=1}^{n}De_{1j})_{1}$ is of the form $(D,(%
\sum_{j=1}^{n}e_{1j}d_{j})D^{\prime })_{1}$ for some $d_{1},\ldots ,d_{n}\in
D$. Moreover, one can also prove that there exists an isomorphism from $%
(D,(\sum_{j=1}^{n}e_{1j}d_{j})D^{\prime })_{1}$ onto $(D,e_{11}D^{\prime
})_{1}$ if and only if there exists a nonzero $d\in D$ such that $%
d(d_{i})_{i=1}^{n}\in \mathfrak{R}$ if and only if there exists $d\in D$ and 
$(x_{i})_{i=1}^{n}\in \mathfrak{R}$ such that $(d_{i})=d(x_{i})_{i=1}^{n}$
if and only if $(d_{i})_{i=1}^{n}\in \bigcup_{(x_{i})\in \mathfrak{R}%
}D(x_{i})$. Now the result follows.\bigskip
\end{proof}

\begin{remark}
\label{rem1}It can be easily seen from Corollary \ref{cor1} (ii) that, for a
division ring $D$ and a division subring $D^{\prime }$ of $\mathbb{M}%
_{2}(D), $ the ring 
\begin{equation*}
\left( 
\begin{array}{cc}
D & 0 \\ 
\begin{array}{c}
D \\ 
D%
\end{array}
& D^{\prime }%
\end{array}%
\right)
\end{equation*}%
has right middle class if and only if either all elements of $D^{\prime }$
are lower triangular matrices or all elements of $D^{\prime }$ are upper
triangular matrices. Suppose, in particular, that all elements of $D^{\prime
}$ are upper triangular matrices. Then for every $[a_{ij}]\in D^{\prime }$, $%
a_{22}$ is uniquely determined. Thus, we have a mapping from $D^{\prime }$
to $D$ such that $[a_{ij}]\mapsto a_{22}$ for all $[a_{ij}]\in D^{\prime }$
which is a ring monomorphism, that is, $D^{\prime }$ can be embedded into $D$
as a ring.
\end{remark}

\begin{example}
\label{ex2}(i) If $D$ is a division ring and $D^{\prime }$ is a division
subring of $\mathbb{M}_{2}(D)$ consisting only of lower triangular matrices,
then the ring $(D,D^{2},D^{\prime })$ is a right Artinian right $SI$ ring
which satisfies the properties $(P1)$ and $(P2).$ However, by Remark \ref%
{rem1}, $(D,D^{2},D^{\prime })$ has right middle class. For instance, if we
let $\delta $ be a derivation on $D$ and consider the division subring 
\begin{equation*}
D^{\prime }=\left\{ \left( 
\begin{array}{cc}
a & \delta (a) \\ 
0 & a%
\end{array}%
\right) \mid a\in D\right\}
\end{equation*}
of $\mathbb{M}_{2}(D),$ then the ring $(D,D^{2},D^{\prime })$ has right
middle class. Thus the converse of Theorem 2 of \cite{Er} is not true, in
general.

(ii) Let $D=\mathbb{Z}_{3}$ and $w=\left( 
\begin{array}{cc}
1 & 2 \\ 
1 & 1%
\end{array}%
\right) .$ Observe that $D^{\prime }=\{0,1,w,w^{2},\ldots ,w^{7}\}$ is a
field. By Remark \ref{rem1}, the ring $(D,D^{2},D^{\prime })$ has no right
middle class.
\end{example}

\begin{example}
\label{sey}Let $p$ be a prime integer, $F\leq F_{1}$ a field extension, and $%
K$ a division subring of $\mathbb{M}_{p}(F)$ which properly contains the
field of scalar matrices in $\mathbb{M}_{p}(F)$. Then the ring%
\begin{equation*}
R=\left( 
\begin{array}{cc}
F_{1}\smallskip & 0 \\ 
F_{1}^{p} & K%
\end{array}%
\right)
\end{equation*}%
has no right middle class. If, in particular, we take $F=%
\mathbb{Q}
,$ then any division subring of $\mathbb{M}_{p}(%
\mathbb{Q}
)$ contains all scalar matrices. It follows that $R$ has no right middle
class for any division subring $K$ of $\mathbb{M}_{p}(%
\mathbb{Q}
)$ which is not the field of scalar matrices.
\end{example}

\begin{proof}
We first claim that for any $i=1,2,\ldots ,p,$ the $i$--th rows of all
elements of $K$ span $F^{p}$ as an $F$--space. To see this, let $A$ be an
element of $K$ which is not a scalar matrix and let $m_{A}(x)$ be the
minimal polynomial of $A$ over $F.$ Let $m_{A}(x)=u(x)v(x),$ where $%
u(x),v(x)\in F[x].$ Then $0=m_{A}(A)=u(A)v(A),$ which implies that one of
the determinants $\det (u(A))$ or $\det (v(A))$ is zero. Assume that $\det
(u(A))=0.$ Since $K$ contains all scalar matrices over $F,$ we have $u(A)\in
K.$ But $K$ is a division ring which means that every nonzero matrix in $K$
has nonzero determinant. This gives that $u(A)=0.$ Since $\deg (u(x))\leq
\deg (m_{A}(x))$ and $m_{A}(x)$ is the monic polynomial of least degree
which assumes $A$ as a root, we must have $u(x)$ and $m_{A}(x)$ are
associates. It follows that $m_{A}(x)$ is irreducible over $F.$ Then the
characteristic polynomial $c_{A}(x)$ of $A$ over $F$ is a power of $m_{A}(x)$%
. This implies that $\deg (m_{A}(x))$ divides $\deg (c_{A}(x))=p.$ Since $p$
is prime, $\deg (m_{A}(x))$ is either $1$ or $p.$ If $\deg (m_{A}(x))=1,$
then $A$ is similar to a scalar matrix, $B$ say. In other words, there
exists a $p\times p$ invertible matrix $P$ over $F$ such that $P^{-1}AP=B.$
Thus $A=PBP^{-1}=BPP^{-1}=B,$ a contradiction. Therefore $m_{A}(x)\in F[x]$
is an irreducible polynomial of degree $p.$

Now we shall show that for any $i=1,2,\ldots ,p,$ the $i$--th rows of the
matrices $I,$ $A,\ldots ,A^{p-1}$ span $F^{p}$ as an $F$--space. In order to
prove this, without loss of generality, we may choose $i=1.$ Assume the
contrary, i. e., the first rows of the matrices $I,$ $A,\ldots ,A^{p-1}$ do
not span $F^{p}.$ Then the first rows of these matrices should be linearly
dependent. So, there exist scalar matrices $c_{0},c_{1},\ldots ,c_{p-1}$
over $F,$ not all zero, such that the first row of the matrix $%
C=c_{0}+c_{1}A+\cdots +c_{p-1}A^{p-1}$ is zero. This gives that $\det (C)=0.$
Since $C$ lies in $K,$ we must have $C=0.$ But then $A$ happens to be a root
of a polynomial over $F$ of degree at most $p-1,$ a contradiction.
Consequently, for any $i=1,2,\ldots ,p,$ the $i$--th rows of all elements of 
$K$ span $F^{p}$ as an $F$--space. This gives that the $F$--space spanned by
the $i$--th rows of elements of $K$ contains the standard basis which is
also contained in the $F_{1}$--space spanned by the $i$--th rows of elements
of $K.$ Therefore, for any $i=1,2,\ldots ,p,$ the $i$--th rows of all
elements of $K$ span $F_{1}^{p}$ as an $F_{1}$--space. This fact is true for
any conjugate of $K$ in $\mathbb{M}_{p}(F)$ because, just as $K,$ it also
properly contains the field of scalar matrices over $F.$ The proof is
complete by Theorem \ref{thm1}.
\end{proof}

\begin{remark}
\label{subdivision ring}Let $p$ be a prime number and $F$ be a field. Let $%
P(x)$ be an irreducible polynomial over $F$ of degree $p$ (if exists) and
let $A\in \mathbb{M}_{p}(F)$ be such that $P(A)=0$ (one can use the
companion matrix of $P(x)$ from linear algebra to find such $A).$ Then the
set $K=\{c_{0}+c_{1}A+\cdots +c_{p-1}A^{p-1}:c_{0},c_{1},\ldots ,c_{p-1}\in
F\}$ is a field isomorphic to $F[x]/(P(x))$ and properly contains the field
of scalar matrices over $F.$ Indeed, we can show that all division subrings $%
K$ of $\mathbb{M}_{p}(F)$ properly containing the field of scalar matrices
are of this form: Let $A$ be an element of $K$ which is not a scalar matrix
and let $B$ be any nonzero element of $K.$ Since the set of the first rows
of $I,$ $A,\ldots ,A^{p-1}$ is a basis for $F^{p},$ we must have, by the
same reasoning used in the proof of Example \ref{sey}, there exist scalar
matrices $c_{0},c_{1},\ldots ,c_{p-1},c_{p}$ over $F,$ not all zero, such
that $c_{0}+c_{1}A+\cdots +c_{p-1}A^{p-1}+c_{p}B=0.$ Here, clearly, $%
c_{p}\neq 0.$ It follows that $B$ is a linear combination of the powers $%
A^{i}$ of $A,$ where $i=0,1,\ldots ,p-1.$ Therefore $K=\{c_{0}+c_{1}A+\cdots
+c_{p-1}A^{p-1}:c_{0},c_{1},\ldots ,c_{p-1}\in F\}.$
\end{remark}

\begin{example}
Let $F$ be a field and $P(x)$ be an irreducible polynomial over $F$ of prime
degree. Then the ring%
\begin{equation*}
\left( 
\begin{array}{cc}
F\medskip & 0 \\ 
F[x]/(P(x)) & F[x]/(P(x))%
\end{array}%
\right)
\end{equation*}%
has no right middle class.
\end{example}

\begin{proof}
Let $\deg (P(x))=p.$ Set $K=\{c_{0}+c_{1}A+\cdots
+c_{p-1}A^{p-1}:c_{0},c_{1},\ldots ,c_{p-1}\in F\},$ where $A$ is the
companion matrix of $P(x)$ over $F.$ It is routine to check that the mapping%
\begin{equation*}
\left( 
\begin{array}{cc}
F\medskip & 0 \\ 
F[x]/(P(x)) & F[x]/(P(x))%
\end{array}%
\right) \longrightarrow \left( 
\begin{array}{cc}
F\smallskip & 0 \\ 
F^{p} & K%
\end{array}%
\right) ,
\end{equation*}%
by%
\begin{equation*}
\left( 
\begin{array}{cc}
a & 0 \\ 
\tsum\limits_{i=0}^{p-1}c_{i}x^{i} & \tsum\limits_{i=0}^{p-1}d_{i}x^{i}%
\end{array}%
\right) \longmapsto \left( 
\begin{array}{cc}
a & 0 \\ 
\begin{array}{c}
c_{0} \\ 
c_{1} \\ 
\vdots \\ 
c_{p-1}%
\end{array}
& \tsum\limits_{i=0}^{p-1}d_{i}A^{i}%
\end{array}%
\right)
\end{equation*}%
is a ring isomorphism. The result follows from Example \ref{sey} and Remark %
\ref{subdivision ring}.
\end{proof}

\begin{example}
\label{example_local module of length 2}The ring 
\begin{equation*}
R=\left( 
\begin{array}{cc}
\mathbb{C}
& 0 \\ 
\begin{array}{c}
\mathbb{C}
\\ 
\mathbb{C}%
\end{array}
& C%
\end{array}%
\right) ,\ \text{where}\ C=\left\{ \left( 
\begin{array}{cc}
a & -b \\ 
b & a%
\end{array}%
\right) \mid a,b\in 
\mathbb{R}
\right\} \cong 
\mathbb{C}
,
\end{equation*}%
has no right middle class. However, $R$ has at least two nonisomorphic local
right $R$--modules of length two by Proposition~\ref{prop1}. Therefore, the
converse of \cite[Proposition 6]{Er} is not true, in general.
\end{example}

Note that if a ring of the form $(D,D^{n},D^{\prime }),$ where $D$ and $%
D^{\prime }$ are as above, has no right middle class, then the property of
having no right middle class of $(D,D^{n},D^{\prime })$ remains unaltered
when we replace $D$ by any division ring containing $D.$ However, such a
replacement may result in an increased number of local modules of length
two. On the other hand, as the following theorem shows, for the ring $%
R=(D,D^{n},D^{\prime })$, it is necessary that certain local $R$--modules of
length two are isomorphic, which is also sufficient when $n=2.$

\begin{lemma}
\label{lem2}Let $1\leq i,j\leq n$. Then $(D,D^{\prime })_{i}\cong
(D,D^{\prime })_{j}$ if and only if $D^{\prime }$ contains an element $A$
such that $A[j,k]=\delta _{ik}c$ for some nonzero $c\in D$.
\end{lemma}

\begin{proof}
The proof is straightforward by Lemma~\ref{lem1}.
\end{proof}

\begin{theorem}
Let $D$ be a division ring and $D^{\prime }$ a division subring of $\mathbb{M%
}_{n}(D).$ If the ring 
\begin{equation*}
R=\left( 
\begin{array}{cc}
D & 0 \\ 
\mathbb{M}_{n\times 1}(D) & D^{\prime }%
\end{array}%
\right)
\end{equation*}%
has no right middle class, then there exists a conjugate $D^{\prime \prime }$
of $D^{\prime }$ in $\mathbb{M}_{n}(D)$ such that $(D,D^{\prime \prime
})_{i}\cong (D,D^{\prime \prime })_{j}$ for each $i,j=1,\ldots ,n$ as right $%
(D,D^{n},D^{\prime \prime })$--modules. In particular, if $n=2,$ then $R$
has no right middle class if and only if $(D,D^{\prime })_{1}\cong
(D,D^{\prime })_{2}.$
\end{theorem}

\begin{proof}
Since the ring $R$ has no right middle class, by Theorem \ref{thm1}, $%
D^{\prime }$ has an element $A$ such that $A[1,n]\neq 0.$ By standard
techniques of linear algebra, it is not difficult to see that there exists a
unit $y_{1}$ in $\mathbb{M}_{n}(D)$ such that $B[1,i]=0$ for all $i=1,\ldots
,n-1$ and $B[1,n]=A[1,n],$ where $B=y_{1}Ay_{1}^{-1}.$ Now, pick an element $%
A^{\prime }$ of \ $D^{\prime }$ such that $A^{\prime }[1,n-1]\neq 0$. Then
there exists a unit $y_{2}\in \mathbb{M}_{n}(D)$ such that $B^{\prime
}[1,i]=0$ for $1\leq i\leq n,$ $i\neq n-1,$ and $B^{\prime
}[1,n-1]=A^{\prime }[1,n-1].$ Notice that conjugating $B$ by $y_{2}$ does
not effect the current form of $B$, i.e., $y_{2}By_{2}^{-1}$ is still a
matrix whose $(1,n)$--th entry is nonzero and $(1,i)$--th entry is zero for
every $i=1,\ldots ,n-1.$ If we continue in this fashion, we may find a unit $%
y$ in $\mathbb{M}_{n}(D)$ such that the ring $yD^{\prime }y^{-1}$ contains
elements $A_{2},\ldots ,A_{n},$ where $A_{i}[1,k]=\delta _{ik}c_{i}$ for
some $0\neq $ $c_{i}\in D$ for all $i=2,\ldots ,n.$ Now the first part of
the theorem follows by Lemma \ref{lem2}. Moreover, the remaining part also
follows from the first part together with Corollary \ref{cor1}.\bigskip
\end{proof}

Our aim in concluding this section is to complete our investigation of rings
which have no right middle class. From Theorem 1 given in the introductory
part, we know exactly what it means for a right $SI$--ring to have no right
middle class . With the following theorem, we determine how precisely rings
with no right middle class which are not right $SI$ look like.

\begin{theorem}
\label{non-SI}Let $R$ be a ring which is not right $SI$. Then $R$ has no
right middle class if and only if $R\cong S\oplus \mathbb{M}_{n}(A),$ where $%
S$ is a semisimple Artinian ring, $n$ is a positive integer, and $A$ is
either zero or a local right Artinian ring whose Jacobson radical properly
contains no nonzero ideals.
\end{theorem}

\begin{proof}
The sufficiency follows from \cite[Corollary 2.14]{lopez-simental}. For the
necessity, suppose that $R$ is a ring with no right middle class which is
not right $SI$. By \cite[Theorem 2]{Er}, $R\cong S\oplus T,$ where $S$ is a
semisimple Artinian ring and $T$ is zero or it is a ring as in Theorem
2(iii) of \cite{Er}. If $T=0,$ then we are done. Let $T$ be nonzero. Then by 
\cite[Corollary 6]{Er}, $T\cong \left( 
\begin{array}{cc}
\mathbb{M}_{n}(A) & 0 \\ 
X & B%
\end{array}%
\right) ,$ where $A$ is a (nonsemisimple) local right Artinian ring, $B$ is
a semisimple Artinian ring, and $X$ is a $B$--$\mathbb{M}_{n}(A)$--bimodule.
Note that $T$ has no right middle class, too. As $J(A)\neq 0,$ by \cite[%
Corollary 2.14]{lopez-simental}, we must have $X=0$ and $J(A)$ properly
contains no nonzero ideals. Since $T$ is indecomposable, we must also have $%
B=0.$ This completes the proof.
\end{proof}

\begin{corollary}
Let $R$ be a right Noetherian ring. Then $R$ has no right middle class if
and only if $R\cong S\oplus T,$ where $S$ is a semisimple Artinian ring and $%
T$ is zero or it is Morita equivalent to one of the following rings:

$(i)$ a right $PCI$--domain, or

$(ii)$ a formal triangular matrix ring of the form $(D,D^{n},D^{\prime })$
where $D$ is a division ring and $D^{\prime }$ is a division subring of $%
\mathbb{M}_{n}(D)$ such that for each conjugate ring $U$ of $D^{\prime }$ in 
$\mathbb{M}_{n}(D)$, the set of $i$--th rows of elements in $U$ span $D^{n}$
as a left $D$--space for every $i=1,\ldots ,n$, or

$(iii)$ a local right Artinian ring whose Jacobson radical properly contains
no nonzero ideals.
\end{corollary}

\begin{proof}
The sufficiency follows from \cite[Proposition 5]{Er}, Theorem \ref{thm1},
and Theorem \ref{non-SI}. For the necessity, assume that $R$ is a right
Noetherian ring which has no right middle class. If $R$ is right Artinian,
then by \cite[Theorem 2]{Er}, Theorem \ref{thm1}, and \ref{non-SI}, $R$ is
Morita equivalent to a ring which belongs to the class of rings in $(ii)$ or 
$(iii).$ If $R$ is not right Artinian, then, by \cite[Theorem 2]{Er}, it is
either Morita equivalent to a right $PCI$--domain or a $V$--ring with
essential socle. Since $R$ is right Noetherian, in the latter case $R$ is
semisimple Artinian. This completes the proof.
\end{proof}

\begin{corollary}
Let $R$ be any ring. Then $R$ has no right middle class if and only if $%
R\cong S\oplus T$ where $S$ is a semisimple Artinian ring and $T$ is zero or
it satisfies one of the following conditions:

$(i)$ $T$ is Morita equivalent to a right $PCI$--domain, or

$(ii)$ $T\cong \mathbb{M}_{n}(A),$ where $n$ is a positive integer, and $A$
is a local right Artinian ring whose Jacobson radical properly contains no
nonzero ideals, or

$(iii)$ $T$ is an indecomposable right $SI$--ring with homogeneous essential
right socle which satisfies one the following equivalent conditions (where $%
Q $ is the maximal right quotient ring of $T$):

\qquad $(a)$ Non-semisimple quasi--injective right $T$--modules are
injective.

$\qquad (b)$ Proper essential submodules of $Q_{T}$ are poor.

\qquad $(c)$ For any submodule $A$ of $Q_{T}$ containing $Soc(T_{T})$
properly, $QA=Q.$
\end{corollary}

\begin{proof}
It follows from \cite[Theorem 2 and Proposition 8]{Er} and Theorem \ref%
{non-SI}.\bigskip
\end{proof}

\begin{remark}
Let $R$ be a nonsemisimple right Artinian ring with no right middle class.
For the sake of simplicity, we assume that $R$ is indecomposable. We know
that either $R$ is a $QF$--ring with Jacobson radical square zero or $R_{R}$
is poor. If $R$ is a $QF$--ring, then it is not right $SI$ (because a
nonsemisimple right $SI$--ring with no right middle class cannot be right
self--injective), and so, by Theorem 2.16, $R\cong \mathbb{M}_{n}(A)$ for
some local right Artinian ring whose Jacobson radical properly contains no
nonzero ideals. Then $A$ is also $QF.$ If $I$ is a nonzero right ideal of $%
R, $ then $J=r.ann(l.ann(I))=I.$ This gives that $cl(R_{R})=2,$ i.e., $R$ is
a right chain ring with right composition length two. However, there are
Artinian rings with no right (and left) middle class which are poor as a
right (and left) module although they have right (or left) composition
length two (see Example \ref{non QF but poor}).
\end{remark}

\begin{example}
\label{non QF but poor}Let $F$ be the field $%
\mathbb{Q}
(\sqrt{2})$ and $\alpha :F\longrightarrow 
\mathbb{Q}
$ defined by $\alpha (a+b\sqrt{2})=a.$ Let $R=F\times F$ as additive abelian
group and define the multiplication on $R$ as%
\begin{equation*}
(u,v)(w,z)=(uw,uz+v\alpha (w)).
\end{equation*}%
Then $R$ is a noncommutative local Artinian ring with $cl(R_{R})=3$ and $%
cl(_{R}R)=2$. Since $cl(_{R}R)=2,$ $J(R)$ properly contains no nonzero
ideals of $R$. It follows that $R$ has no right (and left) middle class.
Also, $R$ does not satisfy the double annihilator condition for right
ideals. Then $R$ cannot be $QF,$ i.e., both $R_{R}$ and $_{R}R$ are
poor.\bigskip
\end{example}

Obviously, a ring with no right middle class is either right self--injective
or poor as a right module over itself. It is known, from \cite[Proposition 9]%
{Er}, that the condition of Theorem 2(iii) in \cite{Er} are sufficient if
the ring is taken to be right self--injective. The next two examples
illustrate that these conditions are not sufficient in general even if the
ring is poor. In particular, these examples also indicate that there are
rings satisfying the conditions of Theorem 2(iii) in \cite{Er} which are not
of the form $\mathbb{M}_{n}(A),$ where $A$ is as in Theorem \ref{non-SI}. We
first need the following lemma.

\begin{lemma}
\label{poor ring}Let $R$ be a right semiartinian ring. Then $R_{R}$ is poor
if and only if $R$ is not injective relative to a local right $R$--module of
length two.
\end{lemma}

\begin{proof}
The necessity is obvious. For the sufficiency, suppose that $R$ is not
injective relative to a local right $R$--module of length two. Let $R$ be $M$%
--injective. Without loss of generality, we may choose $M$ cyclic. Assume
that $M$ is not semisimple. Then there exists a local subfactor $N$ of $M$
of length two since $M$ is semiartinian. This gives that $R$ is $N$%
--injective, a contradiction.
\end{proof}

\begin{example}
\label{sey1}Let $F$ be a field and $V$ a finite dimensional vector space
over $F$ of dimension greater than $1.$ Let $R=\left\{ \left( 
\begin{array}{c}
a\ \ 0 \\ 
v\ \ a%
\end{array}%
\right) :a\in F,\ v\in V\right\} .$ Then $R$ is a commutative local Artinian
ring which satisfies the conditions of Theorem 2$(iii)$ in \cite{Er}.
However, $R$ has right middle class by \cite[Corollary 2.14]{lopez-simental}%
. Note too that $R$ is poor as a module over itself. Indeed, if $L$ is a
local module of length two with simple submodule $S,$ then there is an
isomorphism from $S\ $into $R$ which cannot be extended to a homomorphism
from $L.$ Otherwise, $R$ would contain a local module of length two, which
is impossible since $\dim _{F}(V)>1$. Then $R$ is not $L$--injective. Hence,
by Lemma \ref{poor ring}, $R$ is poor.
\end{example}

\begin{example}
\label{sey2}Let $R=\left( 
\begin{array}{cc}
\mathbb{Z}
/4%
\mathbb{Z}
& 0 \\ 
\mathbb{Z}
/2%
\mathbb{Z}
& 
\mathbb{Z}
/2%
\mathbb{Z}%
\end{array}%
\right) .$ Then

$(1)$ $Soc(R_{R})=J(R)=Z(R_{R}),$

$(2)$ $R$ has essential homogeneous right socle,

$(3)$ there is a unique noninjective simple right $R$--module up to
isomorphism,

$(4)$ $R_{R}$ is poor, and

$(5)$ $R$ has right middle class.
\end{example}

\begin{proof}
$(1)$ It follows from \cite[Proposition 4.2]{G2} that $%
Soc(R_{R})=J(R)=Z(R_{R})=\left( 
\begin{array}{cc}
2%
\mathbb{Z}
/4%
\mathbb{Z}
& 0 \\ 
\mathbb{Z}
/2%
\mathbb{Z}
& 0%
\end{array}%
\right) $

$(2)$ Since $R$ is right Artinian, $Soc(R_{R})\leq _{e}R_{R}.$ On the other
hand, $Soc(R_{R})=\left( 
\begin{array}{cc}
2%
\mathbb{Z}
/4%
\mathbb{Z}
& 0 \\ 
0 & 0%
\end{array}%
\right) \oplus \left( 
\begin{array}{cc}
0 & 0 \\ 
\mathbb{Z}
/2%
\mathbb{Z}
& 0%
\end{array}%
\right) ,$ where simple summands are isomorphic.

$(3)$ Since $R/J(R)\cong S\oplus S^{\prime },$ where 
\begin{equation*}
S=\left( 
\begin{array}{cc}
\mathbb{Z}
/2%
\mathbb{Z}
& 0 \\ 
0 & 0%
\end{array}%
\right) ,\ \text{and}\ S^{\prime }=\frac{\left( 
\begin{array}{cc}
0 & 0 \\ 
\mathbb{Z}
/2%
\mathbb{Z}
& 
\mathbb{Z}
/2%
\mathbb{Z}%
\end{array}%
\right) }{\left( 
\begin{array}{cc}
0 & 0 \\ 
\mathbb{Z}
/2%
\mathbb{Z}
& 0%
\end{array}%
\right) },
\end{equation*}%
any simple right $R$--module is isomorphic to either $S$ or $S^{\prime }.$
As $Soc(R_{R})\leq _{e}R_{R},$ $S_{R}$ cannot be injective. To establish
(3), we shall show that $S_{R}^{\prime }$ is injective. Note that all proper
essential right ideals of $R$ are%
\begin{equation*}
I_{1}=\left( 
\begin{array}{cc}
2%
\mathbb{Z}
/4%
\mathbb{Z}
& 0 \\ 
\mathbb{Z}
/2%
\mathbb{Z}
& 0%
\end{array}%
\right) ,\ I_{2}=\left( 
\begin{array}{cc}
\mathbb{Z}
/4%
\mathbb{Z}
& 0 \\ 
\mathbb{Z}
/2%
\mathbb{Z}
& 0%
\end{array}%
\right) ,\ \text{and }I_{3}=\left( 
\begin{array}{cc}
2%
\mathbb{Z}
/4%
\mathbb{Z}
& 0 \\ 
\mathbb{Z}
/2%
\mathbb{Z}
& 
\mathbb{Z}
/2%
\mathbb{Z}%
\end{array}%
\right) .
\end{equation*}%
If $f:I_{1}\longrightarrow S$ is a nonzero homomorphism of $R$--modules,
then $f(I_{1})=S.$ But $I_{1}\left( 
\begin{array}{cc}
0 & 0 \\ 
0 & 1%
\end{array}%
\right) =0$ while $S\left( 
\begin{array}{cc}
0 & 0 \\ 
0 & 1%
\end{array}%
\right) \neq 0,$ a contradiction. Thus $Hom_{R}(I_{1},S)=0.$ Similarly, $%
Hom_{R}(I_{2},S)=0.$ Now, let $f:I_{3}\longrightarrow S$ be a nonzero
homomorphism. Observe that there are only two maximal right $R$--submodule
of $I_{3}:$ $Soc(R_{R})$ and $M=\left( 
\begin{array}{cc}
0 & 0 \\ 
\mathbb{Z}
/2%
\mathbb{Z}
& 
\mathbb{Z}
/2%
\mathbb{Z}%
\end{array}%
\right) .$ Since $I_{3}/M\ncong S,$ we must have $Ker(f)=Soc(R_{R}).$ It
follows that there exists a unique nonzero homomorphism $f:I_{3}%
\longrightarrow S,$ which can be extended to a homomorphism $%
R\longrightarrow S.$ Therefore, $S_{R}$ is injective.

$(4)$ Note that there are two local right $R$--modules of length two up to
isomorphism: $X=\left( 
\begin{array}{cc}
\mathbb{Z}
/4%
\mathbb{Z}
& 0 \\ 
0 & 0%
\end{array}%
\right) $ and $Y=\left( 
\begin{array}{cc}
0 & 0 \\ 
\mathbb{Z}
/2%
\mathbb{Z}
& 
\mathbb{Z}
/2%
\mathbb{Z}%
\end{array}%
\right) .$ Indeed, we can decompose $R_{R}$ as $R=X\oplus Y.$ Notice that
both $X$ and $Y$ are local right $R$--modules of length two. Now let $M$ be
a local right $R$--module of length two. Then there exists an epimorphism $%
f:R\longrightarrow M.$ Since $M$ is local, $f(X)=M$ or $f(Y)=M.$ This gives
that $X\cong M$ or $Y\cong M.$

$R$ is not $X$--injective because the map $\left( 
\begin{array}{cc}
2%
\mathbb{Z}
/4%
\mathbb{Z}
& 0 \\ 
0 & 0%
\end{array}%
\right) \longrightarrow R\ $defined by $\left( 
\begin{array}{cc}
\overline{2} & 0 \\ 
0 & 0%
\end{array}%
\right) \longmapsto \left( 
\begin{array}{cc}
\overline{2} & 0 \\ 
x & 0%
\end{array}%
\right) ,$ where $0\neq x\in 
\mathbb{Z}
/2%
\mathbb{Z}
,$ is an $R$--homomorphism which does not extend to a homomorphism $%
f:X\longrightarrow R.$ Indeed, if $f\left( 
\begin{array}{cc}
\overline{1} & 0 \\ 
0 & 0%
\end{array}%
\right) =\left( 
\begin{array}{cc}
a & 0 \\ 
y & b%
\end{array}%
\right) ,$ then $f\left( 
\begin{array}{cc}
\overline{2} & 0 \\ 
0 & 0%
\end{array}%
\right) =\left( 
\begin{array}{cc}
2a & 0 \\ 
0 & 2b%
\end{array}%
\right) ,$ a contradiction.

Now we claim that $R$ is not $Y$--injective. To see this, let $f:\left( 
\begin{array}{cc}
0 & 0 \\ 
\mathbb{Z}
/2%
\mathbb{Z}
& 0%
\end{array}%
\right) \longrightarrow R$ with $\left( 
\begin{array}{cc}
0 & 0 \\ 
\overline{1} & 0%
\end{array}%
\right) \longmapsto \left( 
\begin{array}{cc}
\overline{2} & 0 \\ 
\overline{1} & 0%
\end{array}%
\right) .$ Then $f$ is a homomorphism which cannot extend to a homomorphism $%
g:Y\longrightarrow R.$ Indeed, if $g\left( 
\begin{array}{cc}
0 & 0 \\ 
0 & \overline{1}%
\end{array}%
\right) =\left( 
\begin{array}{cc}
a & 0 \\ 
y & b%
\end{array}%
\right) ,$ then $a=y=0$ since $\left( 
\begin{array}{cc}
0 & 0 \\ 
0 & \overline{1}%
\end{array}%
\right) \left( 
\begin{array}{cc}
\overline{1} & 0 \\ 
0 & 0%
\end{array}%
\right) =0$ and $\left( 
\begin{array}{cc}
a & 0 \\ 
y & b%
\end{array}%
\right) \left( 
\begin{array}{cc}
\overline{1} & 0 \\ 
0 & 0%
\end{array}%
\right) =\left( 
\begin{array}{cc}
a & 0 \\ 
y & 0%
\end{array}%
\right) .$ Thus, $g\left( 
\begin{array}{cc}
0 & 0 \\ 
0 & \overline{1}%
\end{array}%
\right) =\left( 
\begin{array}{cc}
0 & 0 \\ 
0 & b%
\end{array}%
\right) .$ But we also have $\left( 
\begin{array}{cc}
0 & 0 \\ 
0 & \overline{1}%
\end{array}%
\right) \left( 
\begin{array}{cc}
0 & 0 \\ 
\overline{1} & 0%
\end{array}%
\right) =\left( 
\begin{array}{cc}
0 & 0 \\ 
\overline{1} & 0%
\end{array}%
\right) $ and $\left( 
\begin{array}{cc}
0 & 0 \\ 
0 & b%
\end{array}%
\right) \left( 
\begin{array}{cc}
0 & 0 \\ 
\overline{1} & 0%
\end{array}%
\right) =\left( 
\begin{array}{cc}
0 & 0 \\ 
\overline{b} & 0%
\end{array}%
\right) ,$ a contradiction.\medskip

By Lemma \ref{poor ring}, $R_{R}$ is poor.\medskip

$(5)$ By \cite[Corollary 2.14]{lopez-simental}, $R$ has right middle class.
\end{proof}

\section{Artinian Rings With No Simple Middle Class}

\qquad Notice that right $V$--rings have automatically no simple middle
class. In the theorem below, we consider right $GV$--rings without simple
middle class whose proof uses almost the same arguments as those used in the
proof of \cite[Lemma 8]{Er}. Before giving the theorem, we need the
following lemma.

\begin{lemma}
\label{l} Suppose that $R$ is not a right $V$--ring. Then $R$ is a right $GV$%
--ring with no simple middle class if and only if $R$ has a simple
projective poor module.
\end{lemma}

\begin{proof}
It follows from \cite[Corollaries 4.4 and 4.5]{AAL}.\bigskip
\end{proof}

As a direct consequence of Lemma~\ref{l}, we have the following lemma.

\begin{lemma}
\label{l1} Suppose that a ring $R$ has no simple middle class. Then $R$ is a
right $GV$--ring or every simple projective right module is injective.
\end{lemma}

The proof of the following theorem uses the notion of orthogonal modules.
Recall that two modules are said to be orthogonal if they have no nonzero
isomorphic submodules.

\begin{theorem}
\label{A} Suppose that $R$ is a right nonsingular right $GV$--ring which is
not a right $V$--ring. If $R$ has no simple middle class, then there is a
ring decomposition $R=S\oplus T$, where $S$ is semisimple Artinian and the
right socle of $T$ is nonzero poor homogeneous. If, further, $R_{R}$ has
finite uniform dimension, then the converse also holds.
\end{theorem}

\begin{proof}
$\mathbf{\mbox{Claim 1}:}$ $Soc(R_{R})$ does not contain a direct sum of two
infinitely generated orthogonal submodules.

Assume that $A$ and $B$ are infinitely generated orthogonal submodules of $%
Soc(R_{R})$. Then $A$ and $B$ are noninjective modules. Let $f$ be a
homomorphism from $E(B)$ into $E(A)$. We will show that $Ker(f)\leq _{e}E(B)$%
. Let $X$ be a nonzero submodule of $E(B)$. Then $X\cap B\neq 0$. If we
assume that $Ker(f)\cap X=0$, then we get $0\neq X\cap B\cong f(X\cap B)$.
But this contradicts the fact that $A$ and $B$ are orthogonal. Hence, $%
Ker(f)\cap X\neq 0$. It follows that $E(B)/Ker(f)\cong Im(f)$ is singular,
and hence $Im(f)=0$. Thus, $A$ is $E(B)$--injective. Since $A$ is poor by 
\cite[Corollary 4.5]{AAL}, $B$ is injective, which is a contradiction.

$\mathbf{\mbox{Claim 2}:}$ One of the two nonisomorphic simple right ideals
is injective.

If $S_{1}$ and $S_{2}$ are two nonisomorphic noninjective simple right
ideals, then $S_{1}$ is $E(S_{2})$--injective which implies by assumption
that either $E(S_{2})$ is semisimple or $S_{1}$ is injective. This gives
that either $S_{1}$ or $S_{2}$ is injective.

By the same technique above, one can observe that a simple right ideal which
is orthogonal to an infinitely generated semisimple right ideal is
injective. Thus, $Soc(R_{R})$ can have only finitely many homogeneous
components. Let $H_{1}$, $H_{2}$, \ldots , $H_{n}$ be the homogeneous
components of $Soc(R_{R})$. Notice that all the $H_{i}$'s will have to be
injective except possibly for at most one of them. If $n=1$ and $H_{1}$ is
noninjective, then $Soc(R_{R})$ is homogeneous and poor, and so we are done.
Now suppose $n>1,$ $H_{1}$ is either noninjective or zero, and $H_{2},\ldots
,H_{n}$ are injective. Set $S=H_{2}\oplus \cdots H_{n}$. Then $R=S\oplus T$
for some right ideal $T$.

$\mathbf{\mbox{Claim 3}:}$ $R=S\oplus T$ is a ring direct sum, where $%
Soc(T_{T})$ is nonzero poor homogeneous.

Obviously, $TS=0$. If $ST\neq 0$, then there exists $s\in S$ such that $%
sT\neq 0$. But $Soc(T_{T})\leq ann_{r}(s)$ and $ST\leq S$. We will show that 
$X=ann_{r}(s)\cap T\leq _{e}T$. Let $I\leq T$ such that $X\cap I=0$. Define $%
f:I\rightarrow R$, $x\mapsto sx$. Clearly $I\cong Im(f)=sI$, and hence $I=0$%
. Therefore, $sT\cong T/(ann_{r}(s)\cap T)$ is zero, which gives that $ST=0$%
. Thus, we obtain a ring decomposition $R=S\oplus T$, where $S$ is
semisimple Artinian and $Soc(T_{T})\cong H_{1}$ is homogeneous, and poor if $%
H_{1}$ is nonzero. It is routine to check that $T$ is a right $GV$--ring,
too. If $Soc(T_{T})$ is zero, then $T$ has to be a right $V$--ring. But this
leads to the fact that $R$ is a right $V$--ring, a contradiction.

For the last statement, suppose that $R_{R}$ has finite uniform dimension
and $Soc(T_{T})$ is poor homogeneous. Then $Soc(T_{T})$ is a direct sum of
finitely many isomorphic simple right ideals. This gives that simple right
ideals of $T$ are poor, by assumption. On the other hand $T$ is also a $GV$%
--ring. Since a simple module is either projective or singular and a simple
projective right $T$--module is isomorphic to a simple right ideal of $T,$
we get that $T$ has no simple middle class. Now the result follows from
Lemma \ref{middle class & ring direct sum}.
\end{proof}

\begin{corollary}
\label{c1} Let $R$ be a nonsemisimple ring with no simple middle class. If $%
R $ is a right semiartinian right $GV$--ring which is not a right $V$--ring,
then we have a ring decomposition $R=S\oplus T$, where $S$ is semisimple
Artinian and $Soc(T_{T})$ is poor homogeneous.
\end{corollary}

\begin{proof}
Since right semiartinian right $GV$--rings are nonsingular the result
follows from Theorem~\ref{A}.
\end{proof}

\begin{corollary}
\cite[Theorem 4.7]{AAL} Let $R$ be a semiperfect right $GV$--ring. If $R$
has no simple middle class, then we have a ring decomposition $R=S\oplus T$,
where $S$ is semisimple Artinian and $T$ is a semiperfect ring with
homogeneous projective and poor right socle.
\end{corollary}

\begin{proof}
Since semilocal right $GV$--rings are right semiartinian, the proof follows
from Corollary~\ref{c1}.
\end{proof}

\begin{proposition}
\label{p1} Let $R$ be a right semiartinian ring with a singular right socle.
If $R$ has no simple middle class, then $R$ is an indecomposable ring with
unique noninjective simple $R$--module up to isomorphism and $Soc(R_{R})$ is
homogeneous.
\end{proposition}

\begin{proof}
Since $Soc(R_{R})\neq 0,$ $R$ is not a $GV$--ring. If a simple right ideal
is injective, then it is projective. But this is a contradiction since
simple right ideals are singular. Consequently, every simple right ideal of $%
R$ is noninjective. Let $S$ be a simple right ideal and $M$ be any
noninjective simple singular right $R$--module. Since $R$ is right
semiartinian, there exists $N\leq E(M)$ such that $M$ is maximal in $N$.
Since $S$ is poor, there exists a homomorphism $f:N\rightarrow E(S)$ such
that $f(N)\nsubseteq S$. Then $S\subsetneq f(N)$. Since $N$ has composition
length two whereas $f(N)$ has at least two, $f$ must be monic. Thus, $S\cong
M$. Noninjective simple modules are not projective because of Lemma \ref{l1}%
, whence they are singular. Thus, $R$ has a unique noninjective simple
module up to isomorphism. Moreover, since $Soc(R_{R})$ is a direct sum of
noninjective simple singular right ideals, it is homogeneous. Also, it is
clear that a semiartinian ring with homogeneous socle is indecomposable.
\end{proof}

\begin{theorem}
\label{B} If $R$ is a right Artinian nonsemisimple ring with no simple
middle class, then $R$ has a ring decomposition $R=S\oplus T$, where $S$ is
semisimple Artinian and $Soc(T_{T})$ is poor homogeneous. Moreover, $%
Soc(R_{R})$ is either projective or singular.
\end{theorem}

\begin{proof}
If $R$ is a right $GV$--ring, then we are done with by Theorem~\ref{A}.
Suppose that $R$ is not a right $GV$--ring. Then every simple projective
module is injective by Lemma~\ref{l1}. Write $Soc(R_{R})=P\oplus N$, where $%
P $ is the sum of all simple projective right ideals of $R.$Then $P$ is
injective. Therefore, $R=P\oplus K$ for some right ideal $K$. Since $N$ does
not have any simple projective summand, $P$ is an ideal of $R$. Now we will
show that $R=P\oplus K$ is a ring direct sum. Obviously, $KP=0$. Assume that 
$PK\neq 0$. Then there exists $p\in P$ such that $pK\neq 0$. But $%
Soc(K)\subseteq ann_{r}(p)$. We claim that $X=ann_{r}(p)\cap K\leq _{e}K$.
Let $I\leq K$ such that $X\cap I=0$. Define $f:I\rightarrow R$ such that $%
x\mapsto px$ for all $x\in I$. Then $I\cong Im(f)=pI$, and hence $I=0$. Note
that $P$ is nonsingular since it is semisimple projective. It follows that $%
pK\cong K/(ann_{r}(p)\cap K)$ is both singular and nonsingular, and hence $%
PK=0$. Thus, $R=P\oplus K$ is a ring direct sum. Then $K\cong R/P$ has no
simple middle class by Lemma~\ref{no middle class & factor rings}. $K$ is
nonzero because $R$ is nonsemisimple. Hence, $Soc(K)$ is nonzero, too. Also, 
$Soc(K)\cong N$ is singular. By Proposition~\ref{p1}, $K$ is an
indecomposable ring with singular poor homogeneous right socle.
\end{proof}

\begin{theorem}
\label{th1} Let $R$ be a right Artinian ring. Then $R$ has no simple middle
class if and only if there is a ring decomposition $R=S\oplus T$ where $S$
is semisimple Artinian and $T$ is zero or has one of the following
properties:

$(1)$ $T$ is a right $SI$--ring with homogeneous right socle.

$(2)$ $T$ has a unique noninjective simple right $R$--module up to
isomorphism, and the right socle of $T$ is (homogeneous) singular.
\end{theorem}

\begin{proof}
($\Rightarrow $) It follows from Lemma \ref{l1}, Theorem~\ref{A},
Proposition~\ref{p1} and Theorem~\ref{B}.

($\Leftarrow $) Let $T$ be a nonzero ring which is not a V--ring and assume
that it satisfies $(1)$. Since $T$ is right Artinian, we have a
decomposition $T=e_{1}T\oplus \ldots \oplus e_{n}T\oplus f_{1}T\oplus \ldots
\oplus f_{k}T$, where $e_{i}$ and $f_{j}$ form a complete set of local
orthogonal idempotents, $e_{i}T$ are isomorphic simple right ideals, and $%
f_{j}T$ are nonsimple local $T$--modules. Since $T$ is right $SI$, the
simple modules of the form $f_{j}T/f_{j}J$ are injective, where $J$ denotes
the Jacobson radical of $T$. Therefore, if a right module does not contain
an isomorphic copy of $e_{i}T$, then it is semisimple. Now assume that $%
e_{i}T$ is $M$--injective, where $M$ is a cyclic right module. Then we have
a decomposition $M=A_{1}\oplus \ldots \oplus A_{p}\oplus B_{1}\oplus \ldots
\oplus B_{q}$, where $A_{k}$ and $B_{t}$ are indecomposable modules such
that the $A_{k}$'s do not contain an isomorphic copy of $e_{i}T$ and the $%
B_{t}$'s contain an isomorphic copy of $e_{i}T$. By the above argument, $%
A_{1}\oplus \ldots \oplus A_{p}$ is semisimple. On the other hand, $e_{i}T$
is $B_{t}$--injective. One can observe that $B_{1}\oplus \ldots \oplus B_{q}$
is semisimple, too. Hence, $e_{i}T$ is poor.

Now assume that $T$ satisfies $(2)$. By assumption, $T$ has no simple
projective module. Then we have a decomposition $T=f_{1}T\oplus \ldots
\oplus f_{m}T$, where $f_{i}T$ are nonsimple local modules. Let $%
f_{i}T/f_{i}J$ be a noninjective module for some $i$. Assume that $%
f_{i}T/f_{i}J$ is $M$--injective for a cyclic module $M$. We can write $%
M=A_{1}\oplus \ldots \oplus A_{p}\oplus B_{1}\oplus \ldots \oplus B_{q}$,
where $A_{k}$ and $B_{t}$ are indecomposable modules such that the $A_{k}$'s
do not contain an isomorphic copy of $f_{i}T/f_{i}J$ and the $B_{t}$'s
contain an isomorphic copy of $f_{i}T/f_{i}J$. Because $T$ has a unique
noninjective simple module up to isomorphism, $A_{1}\oplus \ldots \oplus
A_{p}$ is semisimple. $B_{1}\oplus \ldots \oplus B_{q}$ is also semisimple
since $f_{i}T/f_{i}J$ is $B_{t}$--injective for each $t=1,\cdots ,q$. Hence, 
$T$ has no simple middle class. Now, the theorem follows from Lemma \ref%
{middle class & ring direct sum}.\bigskip
\end{proof}

Following \cite{AAL}, we call a ring $R$ simple--destitute if every simple
right $R$--module is poor. Notice that, just as $V$--rings,
simple--destitute rings also constitute a natural subclass of rings with no
simple middle class. In \cite[Theorem 5.2]{AAL}, it is proved that if a
right Artinian ring $R$ has only one simple module up to isomorphism, then $%
R $ is simple--destitute. Now we establish the converse of this theorem as
follows.

\begin{corollary}
\label{simple-destitute}Assume that $R$ is a right Artinian ring. $R$ is
simple--destitute if and only if either $R$ is semisimple or $R$ has a
unique simple module up to isomorphism.
\end{corollary}

\begin{proof}
$(\Leftarrow )$ It follows from \cite[Theorem 5.2]{AAL}.

$(\Rightarrow )$ If $R$ is semisimple, then we are done. Suppose $R$ is not
semisimple. It follows from \cite[Theorem 5.3]{AAL} that $Soc(R_{R})$ is
singular. Then $R$ is an indecomposable ring by Proposition~\ref{p1}. Since $%
R$ is neither a right $V$--ring nor a right $SI$--ring, we get the desired
result by Theorem~\ref{th1}.\bigskip
\end{proof}

We see, in \cite{AAL} and \cite{Er}, that the ring $S=\left( 
\begin{array}{cc}
F & 0 \\ 
F & F%
\end{array}%
\right) ,$where $F$ is a field, is of a particular interest. In \cite{AAL},
it is shown that $S$ has no simple middle class. In \cite{Er}, Er et al.
proved that $S$ has, indeed, no right middle class. It is also proved, in 
\cite{Er}, that a $QF$--ring $R$ with $J(R)^{2}=0$ and homogeneous right
socle has no right middle class. In the following theorem, we give a more
general result by replacing $QF$ with Artinian serial. Note that the class
of Artinian serial rings contains that of both $QF$--rings of above type and
rings in the form of $S.$

\begin{theorem}
\label{artinian serial}If $R$ is an Artinian serial ring with $J(R)^{2}=0$
and homogeneous right socle, then $R$ has no (simple) middle class.
\end{theorem}

\begin{proof}
Since $R$ is Artinian serial, we can write $R=\oplus _{i=1}^{n}e_{i}R$,
where $e_{i}$'s are local idempotents and $e_{i}R$'s are uniserial. Suppose $%
e_{k}R$ is not simple for some $k=1,\ldots ,n.$ Since $e_{k}J(R)$ is the
unique maximal submodule of $e_{k}R,$ $%
e_{k}J(R)=l.ann_{e_{k}R}(J(R))=Soc(e_{k}R).$ Moreover, $e_{k}R$ is an
injective $R$--module by \cite[13.5, p.124]{Extending Modules}. It follows
that, for each $i=1,\ldots ,n,$ $e_{i}R$ is either a simple module or an
injective local module of length two. Now let $e_{t}R$ and $e_{t^{\prime }}R$
be nonsimple. By homogeneity of the right socle, we have $Soc(e_{t}R)\cong
Soc(e_{t^{\prime }}R).$ Then the injectivity of $e_{t^{\prime }}R$ yields an
isomorphism between $e_{t}R$ and $e_{t^{\prime }}R.$ Thus the nonsimple $%
e_{i}R$'s are all isomorphic to each other.

Now let $M$ be a (simple) module. Assume $M$ is $N$--injective, where $N$ is
cyclic. Since $R$ is an Artinian serial ring, by \cite[Theorem 5.6]{Facchini}%
, $N=\oplus _{k=1}^{r}N_{k}$, where $N_{k}$'s are cyclic uniserial. If $N$
is not semisimple, then there exists $t$ such that $N_{t}$ is not simple.
Since $N_{t}$ is cyclic and local, $N_{t}\cong e_{j}R$ for some $j=1,\ldots
,n$. This gives that $M$ is $e_{j}R$--injective. Also, $M$ is injective
relative to any $e_{i}R$ which is simple. It follows that $M$ is $R$%
--injective, i.e., it is injective. This completes the proof.
\end{proof}

\begin{corollary}
Let $R$ be an indecomposable Artinian serial ring. Then $R$ has no right
middle class if and only if $J(R)^{2}=0$ and $R$ has homogeneous right socle.
\end{corollary}

Theorem \ref{artinian serial} shows that, for a nonsemisimple Artinian
serial ring $R$ with $J(R)^{2}=0$ and homogeneous right socle, $R$ has no
right middle class if and only if $R$ has no simple middle class. However,
one can find an Artinian serial ring with homogeneous right socle and no
simple middle class which has right middle class, as the following example
illustrates.

\begin{example}
\textrm{$Let\ R=\mathbb{Z}/8\mathbb{Z.}$ Then }$R$\textrm{\ is an Artinian
chain ring with no simple middle class by Corollary \ref{simple-destitute}.
However, $R$ has right middle class since $J^{2}(R)\neq 0.$}
\end{example}

\section{Commutative Rings}

In this section, we focus on commutative rings and investigate the property
of having no (simple) middle class. We see that commutative rings with no
middle class are precisely those Artinian rings which decompose into a sum
of a semisimple ring and a ring of composition length two. We start with the
following lemma.

\begin{lemma}
\label{lemma Goodearl}\cite[Exercise 17, Ch. 1, Sec. B]{G2} Let $R$ be a
commutative Noetherian ring, and let $P,$ $Q$ be prime ideals of $R.$ Then $%
P\subseteq Q$ if and only if $Hom_{R}(E(R/P),E(R/Q))\neq 0.$
\end{lemma}

\begin{proposition}
\label{Noether with no middle class -> Artin}If $R$ is a commutative
Noetherian ring with no middle class, then $R$ is Artinian.
\end{proposition}

\begin{proof}
Let $R$ be a commutative Noetherian ring with no middle class. We shall
complete the proof by showing that Krull dimension of $R$ is zero, i.e.,
every prime ideal of $R$ is maximal. If $R$ is a V--ring, then there is
nothing to prove. So, assume that $R$ is not a V--ring. Then there exists a
maximal ideal $P$ of $R$ such that $R/P$ is not injective. Hence $E(R/P)$ is
not semisimple. Let $Q$ be any prime ideal of $R$ such that $Q\neq P.$ Then
by above lemma, $Hom_{R}(E(R/P),E(R/Q))=0,$ and so $R/Q$ is injective
relative to $E(R/P).$ Since $R$ has no middle class and $E(R/P)$ is
nonsemisimple, $R/Q$ is injective. Thus $R/Q$ is a self--injective domain,
which implies that $Q$ is a maximal ideal. This completes the proof.
\end{proof}

\begin{theorem}
A commutative ring $R$ has no middle class if and only if there is a ring
decomposition $R=S\oplus T,$ where $S$ is a semisimple Artinian ring, and $T$
is zero or a local ring whose maximal ideal is minimal.
\end{theorem}

\begin{proof}
Suppose first that $R$ has no middle class. Then, by \cite[Theorem 2]{Er},
there is a ring decomposition $R=S\oplus T$ where $S$ is semisimple Artinian
and $T$ is either zero or fits in one of the following three cases:

Case I : $T$ is Morita equivalent to a right PCI domain $T^{\prime }$. In
this case, since $T^{\prime }$ is right Noetherian, so is $T$. Then by
Proposition \ref{Noether with no middle class -> Artin}, $T$ is Artinian.
Thus $T^{\prime }$ is an Artinian domain, and hence a simple ring. This
gives that $T$ is also a simple ring. Since $T$ is commutative, it is a
field.

Case II : $T$ is an indecomposable SI--ring which is either Artinian or a
V--ring. Assume first that $T$ is Artinian. Then $T$ is a finite product of
local rings. Thus indecomposability gives that $T$ is a commutative local
Artinian ring. Suppose that $T$ is not a field. Then there is a minimal
nonzero ideal $A$ of $T.$ Notice that $A\cong T/J(T).$ But since $T$ is an
SI--ring and $T/J(T)$ is singular as a $T$--module, $A$ is injective. Then $%
A $ is a direct summand of $T,$ which contradicts the indecomposability of $%
T.$ Therefore $T$ is a field. Now let $T$ be a V--ring. We may assume that $%
T $ is not Noetherian. Then by \cite[Lemma 5]{Er}, $T$ is semiartinian. This
gives that $soc(T)\neq 0.$ Let $A$ be a nonzero minimal ideal of $T.$ Since $%
J(T)=0,$ there exists a maximal ideal $\mathfrak{M}$ which does not contain $%
A.$ Then $A\oplus \mathfrak{M}=T.$ It follows that $A=T$ and hence $T$ is a
field.

Case III : $T$ is an indecomposable Artinian ring with $soc(T)=J(T).$ Note
that, just as with Case II, $R$ is a local ring. It, therefore, follows from 
\cite[Corollary 2.14]{lopez-simental} that $T$ is a ring whose maximal ideal 
$J(T)$ is minimal.

Conversely, if $T$ is a commutative local ring whose maximal ideal is
simple, then, clearly, $T$ has a unique (up to isomorphism) local module of
length two (which is, indeed, $T$ itself), and has homogeneous $soc(T)=J(T).$
Thus, by \cite[Proposition 7]{Er}, $T$ has no middle class. Now the result
follows from Lemma \ref{middle class & ring direct sum}.\bigskip
\end{proof}

We give the following immediate consequences of the above theorem.

\begin{corollary}
Any commutative ring with no middle class is Artinian.
\end{corollary}

\begin{corollary}
A commutative ring $R$ is a local ring whose (unique) maximal ideal is
minimal if and only if

(i) $R$ is indecomposable Artinian,

(ii) $soc(R)=J(R),$ and

(iii) $R$ has no middle class.
\end{corollary}

Now we turn our attention to commutative Noetherian rings with no simple
middle class although they need not be Artinian as the following lemma shows.

\begin{lemma}
A commutative local ring (not necessarily Noetherian) has no simple middle
class.
\end{lemma}

\begin{proof}
Let $R$ be a commutative local ring with the unique maximal ideal $\mathfrak{%
M.}$ Let $R/\mathfrak{M}$ be $(R/I)$--injective for some proper ideal $I$ of 
$R.$ Then $R/I$ is a $V$--ring since its unique simple module is injective.
This gives that $I=\mathfrak{M},$ and so $R/\mathfrak{M}$ is a poor module.
This completes the proof.
\end{proof}

\begin{theorem}
Let $R$ be a commutative Noetherian ring. Then $R$ has no simple middle
class if and only if there is a ring decomposition $R=S\oplus T$ where $S$
is semisimple Artinian and $T$ is a local ring.
\end{theorem}

\begin{proof}
The sufficiency follows easily from Lemma \ref{middle class & ring direct
sum} together with the above lemma. For the necessity, let $R$ be a
commutative Noetherian ring with no simple middle class. Suppose $R$ is not
semisimple Artinian. Then $R$ is not a $V$--ring, and so there exists a
maximal ideal $\mathfrak{M}$ of $R$ such that $R/\mathfrak{M}$ is a poor $R$%
--module. Let $\mathfrak{P}$ be a prime ideal of $R$ with $\mathfrak{%
P\nsubseteq M.}$ Then $R/\mathfrak{M}$ is $E(R/\mathfrak{P)}$--injective by
Lemma \ref{lemma Goodearl}. Then $E(R/\mathfrak{P)}$ is semisimple, i.e., $%
E(R/\mathfrak{P)}=R/\mathfrak{P}$ and $\mathfrak{P}$ is a maximal ideal of $%
R.$ Since $R/\mathfrak{P}$ is injective, by \cite[Theorem 3.71]{Lam}, $R_{%
\mathfrak{P}}$ is a field. This, in particular, gives that $\mathfrak{P}$
contains no prime ideals properly, and that $\mathfrak{P}^{k}=\mathfrak{P}$
for every positive integer $k.$ Since $R$ is Noetherian, there exist minimal
prime ideals $\mathfrak{P}_{1},\ldots ,\mathfrak{P}_{n}$ of $R$ such that $%
\mathfrak{P}_{1}^{t_{1}}\cap \ldots \cap \mathfrak{P}_{n}^{t_{n}}=0.$ If $R$
is local, then we are done. So, suppose $R$ is not local. If $\mathfrak{P}%
_{i}$ is contained in $\mathfrak{M}$ for every $i=1,\ldots ,n,$ then for any
maximal ideal $\mathfrak{P}$ of $R,$ $\mathfrak{P}_{1}^{t_{1}}\cap \ldots
\cap \mathfrak{P}_{n}^{t_{n}}=0\subseteq \mathfrak{P}$ which yields $%
\mathfrak{P}_{j}\subseteq \mathfrak{P}$ for some $j.$ Then we must have, by
above arguments, $\mathfrak{P}_{j}=\mathfrak{P=M,}$ a contradiction. Thus we
may arrange the $\mathfrak{P}_{i}$'s in such a way that $\mathfrak{P}%
_{1},\ldots ,\mathfrak{P}_{s}$ are not contained in $\mathfrak{M}$ but $%
\mathfrak{P}_{s+1},\ldots ,\mathfrak{P}_{n}$ are, for some $s<n.$ It follows
that $\mathfrak{P}_{1},\ldots ,\mathfrak{P}_{s},\mathfrak{M}$ is the
complete list of all maximal ideals of $R,$ and that $\mathfrak{P}_{1}\cap
\ldots \cap \mathfrak{P}_{s}\cap \mathfrak{P}_{s+1}^{t_{s+1}}\cap \ldots
\cap \mathfrak{P}_{n}^{t_{n}}=0.$ It is also easy to see that $(\mathfrak{P}%
_{1}\cap \ldots \cap \mathfrak{P}_{s})\oplus (\mathfrak{P}%
_{s+1}^{t_{s+1}}\cap \ldots \cap \mathfrak{P}_{n}^{t_{n}})=R.$ Notice that $%
\mathfrak{P}_{s+1}^{t_{s+1}}\cap \ldots \cap \mathfrak{P}_{n}^{t_{n}}$ is a
semisimple Artinian ring isomorphic to $R/(\mathfrak{P}_{1}\cap \ldots \cap 
\mathfrak{P}_{s})$ whereas $\mathfrak{P}_{1}\cap \ldots \cap \mathfrak{P}_{s}
$ is a local ring isomorphic to $R/(\mathfrak{P}_{s+1}^{t_{s+1}}\cap \ldots
\cap \mathfrak{P}_{n}^{t_{n}}).$ This completes the proof.\medskip 
\end{proof}

\textbf{Acknowledgement:} The authors would like to express their gratitude
to Professor Sergio R. L\'{o}pez-Permouth and the anonymous referee for
their invaluable comments and suggestions which improved the presentation of
this work.

\end{document}